\documentclass[a4paper,11pt]{amsart}

\usepackage[T1]{fontenc}
\usepackage{amsthm}
\usepackage{amssymb}
\usepackage{amsfonts}
\usepackage{mathtools}
\usepackage{mathrsfs}
\usepackage{graphicx}
\usepackage{enumitem}
\usepackage{cancel}
\usepackage{setspace}
\usepackage[english]{babel}
\usepackage{csquotes}
\usepackage{etoolbox}
\usepackage{xcolor}
\usepackage{thmtools}
\usepackage{yfonts}
\usepackage{nameref}
\usepackage{fancyvrb}
\usepackage{lmodern}
\usepackage{microtype}
\usepackage[hidelinks]{hyperref}
\usepackage{cleveref}

\usepackage[
    defernumbers=true,
    maxbibnames=6,
    backend=biber,
    style=alphabetic,
    giveninits=false,
    sortcites=true,
    maxalphanames=4,
    minalphanames=3,
    doi=true,
    isbn=false,
    url=false,
    safeinputenc
]{biblatex}
\usepackage{tikz}
\tikzset{%
  dot/.style={circle, draw, fill=black, inner sep=0pt, minimum width=3pt},
}

\newtheorem{thm}{Theorem}[section]
\newtheorem{lmm}[thm]{Lemma}
\newtheorem{crl}[thm]{Corollary}
\newtheorem{prp}[thm]{Proposition}
\newtheorem{fct}[thm]{Fact}

\theoremstyle{definition}

\theoremstyle{remark}

\newtheorem{qst}[thm]{Question}

\let\phi\varphi
\let\epsilon\varepsilon
\let\emptyset\varnothing

\let\subset\subseteq

\let\nsubset\nsubseteq

\newcommand{\mcal}{\mathcal}
\newcommand{\bb}{\mathbb}

\newcommand{\msf}{\mathsf}
\newcommand{\sr}{\mathscr}

\newcommand{\mrm}{\mathrm}
\newcommand{\fr}{\mathfrak}
\newcommand{\mb}{\boldsymbol}

\newcommand{\la}{\land}				
\newcommand{\emp}{\emptyset}		
\renewcommand{\Cup}{\bigcup}		

\newcommand{\eqv}{\equiv} 				      

\newcommand{\leftrightdelim}[4][]{\def\leftrightdelimparamunused{#1}\ifx\leftrightdelimparamunused\empty\mathopen{}\left#2#4\right#3\mathclose{}\else\mathopen{}\big#2#4\big#3\mathclose{}\fi}
\newcommand{\leftrightseparator}[1][]{\def\leftrightseparatorparamunused{#1}\ifx\leftrightseparatorparamunused\empty~\middle|~\else~\big|~\fi}

\newcommand{\rb}[2][]{\leftrightdelim[#1]{(}{)}{#2}}
\newcommand{\ac}[2][]{\leftrightdelim[#1]{\{}{\}}{#2}}

\newcommand{\card}[2][]{\leftrightdelim[#1]{|}{|}{#2}}
\newcommand{\norm}[2][]{\leftrightdelim[#1]{\lVert}{\rVert}{#2}}

\newcommand{\ab}[2][]{\leftrightdelim[#1]{\langle}{\rangle}{#2}}
\newcommand{\ap}[1]{\text{``}{#1}\text{''}} 		

\newcommand{\set}[3][]{\ac[#1]{#2\leftrightseparator[#1]#3}}

\newcommand{\abset}[3][]{\ab[#1]{#2\leftrightseparator[#1]#3}}

\DeclareRobustCommand{\ssn}{\text{\textswab{s}\kern-.1em$\mathfrak{z}$}}

\newcommand{\srSS}{\sr S\!\!\sr S}
\newcommand{\srOS}{\sr O\!\sr S}
\newcommand{\srED}{\sr E\!\sr D}

\newcommand{\dom}{\mrm{dom}}
\newcommand{\ran}{\mrm{ran}}
\newcommand{\cov}{\mrm{cov}}
\newcommand{\non}{\mrm{non}}
\newcommand{\cof}{\mrm{cof}}
\newcommand{\add}{\mrm{add}}
\newcommand{\omom}{{}^{\omega}\omega}
\newcommand{\fomom}{{}^{<\omega}\omega}

\newcommand{\eqi}{=^\infty}

\newcommand{\neqi}{\mathrel{\cancel{=^\infty}}}
\newcommand{\leqs}{\leq^*}
\newcommand{\leqi}{\leq^\infty}
\newcommand{\nleqs}{\mathrel{\cancel{\leq^*}}}

\newcommand{\ts}{\textstyle}

\newcommand{\acc}{\mrm{acc}}
\newcommand{\osc}{\msf{osc}}
\newcommand{\icoi}{[\omega]^\omega_\omega}

\newcommand{\eqisig}{\eqi_\sigma}
\newcommand{\neqisig}{\mathrel{\neqi\hspace{-4mm}{}_{\sigma}\hspace{1mm}}}

\title{Subseries Numbers for Convergent Subseries}
\author{Tristan van der Vlugt}
\thanks{This research was funded by the Austrian
Science Fund (FWF), grants 10.55776/P33895 and 10.55776/PAT8484324.}
\address{Institut für Diskrete Mathematik und Geometrie, TU Wien, Wiedner Hauptstrasse 8-10/104, 1040, Wien, Austria}
\email{tristan@tvdvlugt.nl}


\VerbatimFootnotes

\begin{document}

\begin{abstract}
    Every conditionally convergent series of real numbers has a subseries that diverges. The subseries numbers, previously studied in \cite{Subseries}, answer the question how many subsets of the natural numbers are necessary, such that every conditionally convergent series has a subseries that diverges, with the index set being one of our chosen sets. By restricting our attention to subseries generated by an index set that is both infinite and coinfinite, we may ask the question where the subseries have to be convergent. The answer to this question is a cardinal characteristic of the continuum. 
    
    We consider several closely related variations to this question, and show that our cardinal characteristics are related to several well-known cardinal characteristics of the continuum. In our investigation, we simultaneously will produce dual results, answering the question how many conditionally convergent series one needs such that no single infinite coinfinite set of indices makes all of the series converge.
\end{abstract}

\maketitle


\section{Introduction}

Given a conditionally convergent series $\sum_{i\in\omega}a_i$, we can always find some infinite set $X\subset\omega$ of indices such that the subseries $\sum_{i\in X}a_i$ no longer converges. Sparked by a question from Mohammadpour on MathOverflow,\footnote{See: \href{https://mathoverflow.net/questions/221106/}{https://mathoverflow.net/questions/221106/}}  Brendle,   Brian and Hamkins \cite{Subseries} wrote an article answering the question how large a family $\mcal X\subset[\omega]^\omega$ of infinite sets of natural numbers should be such that for every conditionally convergent series $\sum_{i\in \omega}a_i$ there exists some $X\in\mcal X$ for which $\sum_{i\in X}a_i$ is no longer convergent. The minimal cardinality of such a family $\mcal X$ is called the \emph{subseries number}. It turns out that the subseries number is a cardinal characteristic of the continuum, that is, a cardinality between $\aleph_1$ and $2^{\aleph_0}$ whose size is consistently different from either bound, and hence independent from $\msf{ZFC}$. 
In their article, Brendle, Brian and Hamkins consider some variations of the subseries number depending on the convergence behaviour. For instance, they consider he number of subsets needed such that every conditionally convergent series can be made to diverge by tending to infinity.

A closely related family of cardinal characteristics is the family of \emph{rearrangement numbers}, studied by Blass, Brendle, Brian, Hamkins, Hardy and Larson in \cite{Rearrangement}. The rearrangement number answers a similar question: how many permutations $\pi$ of $\omega$ are necessary such that for each conditionally convergent series $\sum_{i\in\omega}a_i$ the permutation of the indices results in a series  $\sum_{i\in\omega}a_{\pi(i)}$ that no longer converges to the same limit. Note that Riemann's rearrangement theorem \cite{Riemann} states that any conditionally convergent series can be rearranged to converge to a different limit or to diverge, which makes this notion well-defined.

One particular variation on this theme, is the rearrangement number $\fr{rr}_f$, being the least cardinality of a set $P$ of permutations on the natural numbers, such that for every conditionally convergent series $\sum_{i\in\omega}a_i$ there is a permutation $\pi\in P$ for which $\sum_{i\in\omega}a_{\pi(i)}$ is still convergent, but has a different limit than the original series. It is noted in the paper on the subseries number that a similar subseries number is not very interesting, as it is easily observed that for every conditionally convergent series $\sum_{i\in\omega}a_i$ there is some $n\in\omega$ such that $\sum_{i\in \omega\setminus \ac n}a_i$ converges to a different limit (namely, by letting $n$ be any index such that $a_{n}\neq 0$), and thus the countable set $\mcal X=\set{\omega\setminus \ac n}{n\in\omega}$ is a witness for this potential variation on the subseries number.

In our article we will show that there are ways to define a subseries number such that questions related to convergent subseries produce interesting cardinal characteristics. Our deviation from the original article, is that we will not consider all infinite subsets of $\omega$, but only those subsets that are also coinfinite, that is, whose complement is infinite as well.

We will also investigate dual cardinals to the subseries numbers. These dual subseries numbers answer questions such as: how many conditionally convergent series are needed such that no single infinite subset of $\omega$ makes all of them divergent? As with the subseries numbers, the dual subseries numbers are usually cardinal characteristics of the continuum. Relational systems and Tukey connections between them allow us to show relations between cardinal characteristics and their duals simultaneously, so we will prove our relations in this framework where possible.

\subsection*{Outline of the Paper}

After we introduce our notation in \Cref{sec:preliminaries}, as well as a reminder about relational systems and Tukey connections, we briefly discuss in \Cref{sec:coinfinite harmless}  why the change from considering infinite coinfinite subsets of $\omega$, instead of just infinite subsets, does not affect the cardinality of the original subseries numbers defined in \cite{Subseries}. 

In \Cref{sec:convergent subseries}, we will investigate the cardinals $\ssn_{c}$, $\ssn_{cc}$ and $\ssn_{ac}$, defined as subseries numbers for convergent subseries, conditionally convergent subseries, and absolutely convergent subseries respectively. We will show that $\cov(\mcal M)\leq \ssn_c\leq\ssn_{cc}\leq\ssn_{ac}=\fr d$ and that $\fr s\leq\ssn_{cc}$. 

In \Cref{sec:specific limits} we consider subseries numbers $\ssn_e$ and $\ssn_l^A$ that are related to subseries that converge to a specific fixed (set of) limit(s). We provide some sets of limits such that these cardinals are provably equal to either $\fr c$ or $\ssn_c$.

In \Cref{sec:unconditional series} we consider a subseries number $\ssn_n$ related to subseries of conditionally convergent series that are still conditional, but not necessarily convergent. We show that $\ssn_n=3$ for rather pathological reasons. Additionally we consider a subseries number $\ssn_c^{ui}$ of subseries that diverge by tending to infinity unconditionally, and show that it equals $\fr d$.

Finally, in \Cref{sec:dual bonus} we give an overview of the dual cardinals $\ssn_{i,o}^\bot$, $\ssn_o^\bot$ and $\ssn_i^\bot$ to the original subseries numbers from \cite{Subseries}. We modify one of the original proofs in order to show that the dual result holds.


\section{Preliminaries}\label{sec:preliminaries}

The focus of this article is on cardinal characteristics of the continuum, which are generally studied in the context of set theory. We will therefore try to follow the notation from standard set-theoretic literature, such as \cite{Jech} or \cite{Kunen}.

We will denote the set of real numbers by $\bb R$, the set of rational numbers by $\bb Q$ and the set of natural numbers by $\omega$. We adopt the convention that $0$ is a natural number and identify each natural number $n\in\omega$ with the set $n=\ac{0,\dots,n-1}$ of natural numbers strictly below $n$. For real numbers $c\in\bb R$ we let $|c|$ denote the absolute value of $c$. We also write $\bb R_{>0}$ for the set of positive real numbers.

If $X,Y$ are sets, we write ${}^XY$ for the set of functions $X\to Y$ and we define ${}^{<\omega}Y$ as $\Cup_{n\in\omega}{}^n Y$. If $\kappa$ is a cardinality, then $[X]^\kappa$ is the set of subsets of $X$ of size $\kappa$, and $[X]^{<\kappa}$ is the set of subsets of $X$ of size strictly smaller than $\kappa$. We write $[X]^\omega$ for the subsets of $X$ of size $|\omega|=\aleph_0$, and note that if $X$ is countable, then $[X]^\omega$ is the set of all infinite subsets of $X$. If $X$ is countably infinite, then we say that $x\in[X]^\omega$ is \emph{coinfinite} if $X\setminus x\in[X]^\omega$ as well. We introduce the notation $[X]^\omega_\omega$ for set of \emph{infinite coinfinite} subsets of $X$.
We let $A\mathrel\triangle B=(A\setminus B)\cup (B\setminus A)$ denote the symmetric difference between $A$ and $B$.

\subsection{Infinite Series}

For sequences of real numbers, we use boldface to denote the entire sequence, and lightface with subscript indices to denote individual terms. For example, we write $\mb a=\abset{a_i}{ i\in\omega}$ for a sequence in ${}^\omega\bb R$. Let $\fr S\subset{}^\omega\bb Q$ be the set of all sequences $\mb a\in{}^\omega\bb Q$ such that $\mb a$ vanishes at infinity, i.e.\ for each $\epsilon>0$ there is $n\in\omega$ such that $|a_i|<\epsilon$ for all $i\geq n$.

Given an increasing enumeration $\abset{i_n}{ n\in\omega}$ of some $X\in[\omega]^\omega$, we will define the restriction $\mb a\restriction X=\abset{a_{i_n}}{ n\in \omega}$. It is clear that if $\mb a\in\fr S$, then $\mb a\restriction X\in\fr S$ as well.
For $\mb a\in\fr S$, $X\in[\omega]^\omega$ and $k\in\omega$, we will write the partial sum of all elements in $X$ smaller than $k$ as:
\begin{align*}
	\ts\sum_{X\cap k}\mb a=\underset{\substack{i\in X\\i<k}}{\sum}a_i.
\end{align*}
In order to discuss the convergence of sequences $\mb a\in\fr S$, we study the accumulation points of the sequence of partial sums $\mb x=\abset{\sum_k\mb a}{ k\in\omega}$, where we note that $\mb x\in{}^\omega\bb R$. A real number $c\in\bb R$ is an \emph{accumulation point} of a sequence $\mb x\in{}^\omega\bb R$ if for every $\epsilon>0$ there are infinitely many $n\in\omega$ with $|x_n-c|<\epsilon$. Moreover, we say $\infty$ (respectively $-\infty$) is an \emph{accumulation point} of $\mb x$ if for every $k\in\omega$ there is $n\in\omega$ such that $x_n>k$ (respectively $-x_n>k$). Let $\bb R_\infty=\bb R\cup\ac{\infty,-\infty}$. Given $\mb x\in{}^\omega\bb R$, we define $\acc(\mb x)\subset\bb R_\infty$ to be the set of all accumulation points of $\mb x$. 

Given $\mb a\in\fr S$, if $\acc(\abset{\sum_{k}\mb a}{ k\in\omega})=\ac{c}$ for some $c\in\bb R_\infty$, then we say that $c$ is the \emph{limit of the series} $\sum\mb a$; in other words, $c$ is the unique accumulation point of the sequence of partial sums of $\mb a$. In this case, we will write $\sum\mb a=c$ for notational convenience and say that $\sum\mb a$ \emph{tends to a limit}. If, on the other hand, $\acc(\abset{\sum_{k}\mb a}{ k\in\omega})$ is not a singleton, we say that \emph{the series} $\sum\mb a$ \emph{has no limit} or that $\sum\mb a$ \emph{diverges by oscillation}, and for notational convenience we write $\sum\mb a=\osc$. Here $\osc$ is merely a notational convenience and does not indicate a specific value of the infinite series $\sum \mb a$. 

If $\mb a\in\fr S$ and $X\subset\omega$, we define the following notation: 
\begin{align*}
	\ts\sum_X\mb a&=\ts\sum a\restriction X,\\
	P_{\mb a}&=\set{i\in\omega}{a_i>0},\\
	N_{\mb a}&=\set{i\in\omega}{a_i<0},\\
	Z_{\mb a}&=\set{i\in\omega}{a_i=0}.
\end{align*}
We will define several types of convergence and divergence. Let $\mb a\in\fr S$, then $\sum\mb a$ ...
\begin{itemize}[label={...}]
	\item is \emph{convergent} if it tends to a limit in $\bb R$, and \emph{divergent} otherwise.
	\item is \emph{conditional} if $\sum_{P_{\mb a}}\mb a=\infty$ and $\sum_{N_{\mb a}}\mb a=-\infty$.
	\item is \emph{conditionally convergent} if it is both convergent and conditional.
	\item is \emph{absolutely} (or \emph{unconditionally}) \emph{convergent} if it is convergent and not conditional.
	\item \emph{tends to infinity} if it tends to a limit in $\ac{\infty,-\infty}$, that is, $\sum\mb a=\pm\infty$.
	\item is \emph{divergent by oscillation} if has no limit, that is, $\sum\mb a=\osc$.
\end{itemize}
Under our notation, we recall the well-known facts that a series $\mb a$ is divergent if and only if $\sum \mb a=\pm\infty$ or $\sum\mb a=\osc$, and that $\sum\mb a$ is absolutely convergent if and only if $\sum|\mb a|$ is convergent, where $|\mb a|=\abset{|a_i|}{ i\in\omega}$ denotes the absolute values of $\mb a$.

Each of these forms of convergence or divergence, provides us with a subset of $\fr S$.
\begin{align*}
	\fr S_c&=\set{\mb a\in\fr S}{ \ts\sum\mb a\text{ is convergent}},\\
	\fr S_{cc}&=\set{\mb a\in\fr S}{ \ts\sum\mb a\text{ is conditionally convergent}},\\
	\fr S_{ac}&=\set{\mb a\in\fr S}{ \ts\sum\mb a\text{ is absolutely convergent}},\\
	\fr S_i&=\set{\mb a\in\fr S}{ \ts\sum\mb a\text{ tends to infinity}},\\
	\fr S_o&=\set{\mb a\in\fr S}{ \ts\sum\mb a\text{ diverges by oscillation}}.
\end{align*} 
We will also fix the convention for the above classes that multiple subscripts separated by commas are the union of the individual classes. For example, $\fr S_{o,i}=\fr S_o\cup\fr S_i$ is the set of all $\mb a\in\fr S$ that are divergent.
The following proposition is trivial to check.

\begin{prp}\label{lmm:decomposition}
	We have $\fr S=\fr S_{c,i,o}$ and $\fr S_c=\fr S_{cc,ac}$, and note that all unions are in fact disjoint unions.
\end{prp}

\subsection{Relational Systems}\label{sec:relational systems}

In this section we will introduce a very powerful tool that will aid us in providing an order (of cardinality) between certain cardinal characteristics. Most of the cardinal characteristics discussed in this article can be described as the norm of a relational system, and the existence of a morphism between relational systems implies an order between the norms. We refer to \cite[\S~4]{Blass} for a detailed exposition on relational systems.

A \emph{relational system} is a triple $\sr X=\ab{X,Y,R}$ consisting of sets $X,Y$ and a relation $R\subset X\times Y$. We call $X$ the set of \emph{challenges} and $Y$ the set of \emph{responses}, and we say that a response $y\in Y$ \emph{meets} a challenge $x\in X$ if $(x,y)\in R$. 
The \emph{norm} $\norm{\sr X}$ of $\sr X$ is defined as the least cardinality of a set $B\subset Y$ of responses such that any challenge is met by at least one response from $B$. The \emph{dual} $\sr X^\bot$ of $\sr X$ is defined as the triple $\ab{Y,X,R^\bot}$, where 
\begin{align*}
	R^\bot=\set{(y,x)\in Y\times X}{ (x,y)\notin R}.
\end{align*}
Note that, with respect to $\sr X$, we can describe the norm $\norm{\sr X^\bot}$ as the least cardinality of a set of challenges $A\subset X$ such that no single response will meet all of the challenges from $A$.

A \emph{Tukey connection} is a pair of maps between relational systems $\sr X=\ab{X,Y,R}$ and $\sr X'=\ab{X',Y',R'}$, specifically \begin{align*}
	&\rho_-:X\to X',\\
	&\rho_+:Y'\to Y,
\end{align*}
such that if $(\rho_-(x),y')\in R'$ for $x\in X$ and $y'\in Y'$, then $(x,\rho_+(y'))\in R$. The existence of a Tukey connection is written as $\sr X\preceq \sr X'$. Two systems $\sr X$ and $\sr X'$ are \emph{Tukey equivalent} if both $\sr X\preceq\sr X'$ and $\sr X'\preceq\sr X$, which we will abbreviate as $\sr X\equiv\sr X'$. The existence of a Tukey connection provides us with an ordering between the norms of the relational systems.

\begin{prp}
	If $\sr X\preceq\sr X'$, then $\norm{\sr X}\leq\norm{\sr X'}$ and $\norm{\sr X'^\bot}\leq\norm{\sr X^\bot}$.
\end{prp}

For our subseries numbers, we define the relation $S_{(\cdot)}\subset\fr S\times\icoi$  where $(\mb a,X)\in S_{(\cdot)}$ if and only if $\mb a\restriction X\in\fr S_{(\cdot)}$. For example, $(\mb a,X)\in S_i$ if and only if $\sum_X\mb a$ tends to infinity. We define the relational systems
\begin{align*}
	\srSS^{(*)}_{(\cdot)}=\ab{\fr S_{(*)},\icoi,S_{(\cdot)}}\text{ with norm }\norm{\srSS_{(\cdot)}^{(*)}}=\ssn^{(*)}_{(\cdot)}.
\end{align*}
The symbol $\ssn$ used for the cardinality of the norm is the German letter \emph{eszett}, used in most of the German dialects as a ligature for a \emph{double s}. Since we will usually have the set $\fr S_{cc}$ as our set of challenges and never use $\fr S$ as our set of challenges, we will write $\ssn_{(\cdot)}$ instead of  $\ssn^{cc}_{(\cdot)}$. This also agrees with notation from earlier papers, where the cardinals $\ssn_i$, $\ssn_{o}$, and $\ssn_{i,o}$ were written as $\text{\ss}_i$, $\text{\ss}_{o}$, and $\text{\ss}_{io}$, respectively. Our deviation from previous notation is to use a fraktur font for the cardinal characteristics (as is common), and to separate the subscripts with commas (to avoid ambiguity).%
\footnote{As far as the author is aware, there is no \emph{eszett} present in the existing \texttt{\textbackslash mathfrak} alphabet. Nevertheless, a reasonably close approximation is achievable by modifying the kerning between the fraktur letters $\mathfrak f$ and $\mathfrak z$: ``\texttt{\textbackslash mathfrak f\textbackslash kern-.15em\textbackslash mathfrak z}'' produces $\mathfrak f\kern-.2em\mathfrak z$. Of course, instead of $\mathfrak f$, one should really be using a fraktur font supporting the \emph{medial s} ( \scalebox{.7}{\textswab{s}} ), yielding the almost unnoticably superior \scalebox{.7}{\textswab{s}}\kern-.13em$\fr z$. However, obtaining a \emph{medial s} comes at the cost of having to import an additional package, such as \texttt{yfonts}.} %
%
%
%
Mathematically speaking, the definition of the cardinals is also slightly different, but we will argue in \Cref{sec:coinfinite harmless} that this difference is inconsequential.

We can make some simple observations about relational systems. Note that if $X\subset X'$, $Y'\subset Y$ and $R'\subset R\subset X'\times Y$, then for the relational systems $\sr X=\ab{X,Y,R}$ and $\sr X'=\ab{X',Y',R'}$ we have $\sr X\preceq\sr X'$, as witnessed by the Tukey connection consisting of identity functions. Consequently, decompositions such as those given in \Cref{lmm:decomposition} immediately give us a relation between the relevant subseries numbers. Specifically we observe:
\begin{fct}\label{trivial inference}
	$\srSS_{(\cdot\cdot)}^{(*)}\preceq\srSS_{(\cdot)}^{(**)}$ holds whenever $\fr S_{(\cdot)}\subset \fr S_{(\cdot\cdot)}$ and $\fr S_{(*)}\subset \fr S_{(**)}$.
\end{fct}

\section{Why the Old Cardinals are Unchanged}\label{sec:coinfinite harmless}

Before we investigate the new subseries numbers related to convergence, we will shortly discuss the original subseries numbers from \cite{Subseries}, and the minor changes that we have made. We argue that these changes are harmless, and do not lose any generality with regard to the results from the original paper.

In our paper we have defined $\fr S$ to be the set of sequences of rational numbers (i.e.\ ${}^\omega\bb Q$), instead of the more general sequences of real numbers (i.e.\  ${}^\omega\bb R$). 
We remark that if $\mb a\in{}^\omega\bb R$ then there exists $\mb a'\in{}^\omega\bb Q$ and $\mb c\in{}^\omega\bb R$ such that $\sum\mb c$ is absolutely convergent and $a_i=a'_i+c_i$ for each $i\in\omega$. Since the relevant convergence behaviours (conditional convergence, tending to infinity and divergence by oscillation) are all invariant under termwise addition with an absolutely convergent series, it follows that $\mb a$ and $\mb a'$ have equivalent convergence behaviour. Therefore, we lose no generality by restricting our attention to sequences of rational terms.

Now let us discuss another difference between our subseries numbers and those from the original paper. The original subseries numbers can be defined as the norms of the following three relational systems:
\begin{align*}
	\srOS_i^{cc}&=\ab{\fr S_{cc},[\omega]^\omega,S_i}\\
	\srOS_o^{cc}&=\ab{\fr S_{cc},[\omega]^\omega,S_o}\\
	\srOS_{i,o}^{cc}&=\ab{\fr S_{cc},[\omega]^\omega,S_{i,o}}
\end{align*}

That is, the sets of responses consists of infinite sets, $[\omega]^\omega$, whereas this paper uses infinite \emph{coinfinite} sets, $\icoi$, as the sets of responses. It is easy to see that these relational systems are Tukey equivalent to the systems $\srSS_i^{cc}$, $\srSS_o^{cc}$ and $\srSS_{i,o}^{cc}$, respectively. 
To see this, note that if $\sum\mb a$ is convergent and $X$ is a cofinite subsets of $\omega$, then $\sum_X\mb a$ is still convergent. Therefore, the only responses $X\in [\omega]^\omega$ for which a challenge $\mb a$ can be met (that is, $\sum_X\mb a=\pm\infty$ or $\sum_X\mb a=\osc$, depending which relational system we work with), are those responses that were coinfinite to begin with. 

It is exactly this point that makes relational systems of the original type with respect to convergent series, such as $\srOS_c^{cc}=\ab{\fr S_{cc},[\omega]^\omega,S_c}$, rather uninteresting as well: the singleton $\ac{\omega}$ clearly witnesses for trivial reasons that every conditionally convergent series $\sum\mb a$ has a response $X\in\ac{\omega}$ such that $\sum_X\mb a$ is still convergent.
This triviality disappears when we study relational systems with responses from $\icoi$, such as the system $\srSS_c^{cc}$ we defined in \Cref{sec:relational systems}.

\section{Convergent Subseries}\label{sec:convergent subseries}

In this section we will take a look at the norm $\ssn_c$ of $\srSS_c^{cc}$, or, the least cardinality of a family $\mcal X\subset\icoi$ such that for all sequences $\mb a\in\fr S_{cc}$ there is some $X\in\mcal X$ for which $\sum_X\mb a$ is convergent. We will also discuss the closely related cardinals $\ssn_{cc}$ and $\ssn_{ac}$, which need the subseries $\sum_X\mb a$ to be conditionaly and absolutely convergent, respectively. 

We trivially deduce (as in \Cref{trivial inference}) that $\ssn_c$ is the smallest of these cardinals, and analogously $\ssn_c^\bot$ is the largest of the dual cardinal characteristics.

\begin{prp}
	$\ssn_c\leq\min\ac{\ssn_{cc},\ssn_{ac}}$ and $\ssn_c^\bot\geq\max\ac{\ssn_{cc}^\bot,\ssn_{ac}^\bot}$.
\end{prp}

It is not hard to see that $\ssn_c$ is uncountable. Indeed, if $\mcal X\subset\icoi$ is a countable family, say, $\mcal X=\set{X_n}{ n\in\omega}$, then we may find infinitely many disjoint intervals $I_{n,k}$ for $n,k\in\omega$ such that each $I_{n,k}$ contains a strictly increasing sequence $\abset[b]{i_{n,k}^m}{ m<2k}$ with $i_{n,k}^m\in X_n$ if and only if $m$ is even. We then define 
\[
a_i=\begin{cases}
	\frac{(-1)^m}k & \text{if $i=i_{n,k}^m$ for some $n,k\in\omega$ and $m<2k$,}\\
	0 & \text{otherwise.}
\end{cases}
\]
One can quickly see that $\sum_{X_n}\mb a$ diverges, since $\sum_{X_n\cap I_{n,k}}\mb a=1$ for each $k\in\omega$. Therefore $\mcal X$ is too small to witness $\ssn_c$. A more elaborate version of this argument in fact shows that $\cov(\mcal M)\leq\ssn_c$, as we will see in the next subsection.

\subsection{Covering Number of the Meagre Ideal}\label{sec:covM and ssc}
Let us introduce the meagre ideal and the related cardinal characteristics $\cov(\mcal M)$ and $\non(\mcal M)$.
Remember that a subset of the Baire space $X\subset\omom$ is called nowhere dense if every nonempty open $A\subset\omom$ contains a nonempty open $B\subset A$ such that $B\cap X=\emp$. A subset $X\subset\omom$ is meagre if it is the countable union of nowhere dense sets. The set of meagre subsets of $\omom$ is a $\sigma$-ideal, and we will denote it by $\mcal M$.

The \emph{covering number} of the meagre ideal $\cov(\mcal M)$ is the least size of a family $M\subset\mcal M$ such that $\Cup M=\omom$. The \emph{uniformity number} of the meagre ideal $\non(\mcal M)$ is the least size of a subset of $\omom$ that is not meagre. Note that $\sr{C}_\mcal M=\ab{\omom,\mcal M,\in}$ is a relational system with $\norm{\sr C_\mcal M}=\cov(\mcal M)$ and $\norm{\sr C_\mcal M^\bot}=\non(\mcal M)$. However, for our purpose of showing that $\cov(\mcal M)\leq\ssn_c$, it will be convenient to use an alternative relational system $\srED_\sigma$ for which we can show that it has the same norm as $\sr C_\mcal M$.

A well-known theorem of Bartoszy\'nski \cite[Theorems 2.4.1 \& 2.4.7]{BartoszynskiJudah} shows that the norms of the relational system $\srED=\ab{\omom,\omom,\neqi}$ are $\norm{\srED}=\cov(\mcal M)$ and $\norm{\srED^\bot}=\non(\mcal M)$, where $x\eqi y$ if and only if there are infinitely many $k\in\omega$ with $x(k)=y(k)$.\footnote{Note that it is easy to show $\sr C_\mcal M\preceq\srED$, but that $\srED\npreceq\sr C_\mcal M$, because Zapletal \cite{ZapletalCohen} proved that it is possible to add an infinitely equal generic real with a proper forcing notion that does not add a Cohen real. Therefore, the proof that $\norm{\srED}\leq\cov(\mcal M)$ is not via a Tukey connection.} We define a stronger relation $\eqisig$ based on a function $\sigma:\fomom\to\omega$, where $x,y\in\omom$ have $x\eqisig y$ if and only if there are infinitely many $k\in\omega$ such that $x(k)=y(k)$ and $\sigma(x\restriction k)=\sigma(y\restriction k)$. We will call the function $\sigma:\fomom\to\omega$ \emph{suitable} if for every $s\in\fomom$ and $x\in\omom$ there is some $k\in\omega$ and $t\in{}^{k}\omega$ with $s\subset t$ such that $\sigma(t)=\sigma(x\restriction k)$. 

\begin{figure}[h]\small\centering
	\begin{tikzpicture}[xscale=-1]
		\node[dot] (s) at (2.2,1.5) {};
		\node[dot] (t) at (2.9,2.5) {};
		\node[dot] (xn) at (1.47,2.5) {};
		
		\node at (2.45,1.5) {$s$};
		\node at (2.9,2.8) {$t$};
		\node at (1.9,2.8) {$x\restriction k$};
		\node at (.5,2.5) {$k$};
		\node at (1,4.3) {$x$};
		\node at (2.4,3.6) {$\sigma(t)=\sigma(x\restriction k)$};
		
		\draw 
		(0,4) edge (2,0)
		(2,0) edge (4,4)
		(2,0) edge[out=85, in=-80,looseness=1.5] (s)
		(s) edge[out=100, in=-120,looseness=1.5] (t)
		(2,0) edge[out=100, in=-85,looseness=1.5] (1.6,2)
		(1.6,2) edge[out=95,in=-100,->] (1,4)
		(.75,2.5) edge[dashed] (3.25,2.5);
	\end{tikzpicture}
	\caption{A schematic representation of the definition of \emph{suitable} $\sigma:\fomom\to\omega$}
\end{figure}
We define the relational system $\srED_\sigma=\ab{\omom,\omom,\neqisig}$ for suitable $\sigma:\fomom\to\omega$. 

\begin{lmm}\label{lmm:cm eds ed}
	If $\sigma:\fomom\to\omega$ is suitable, then $\sr C_\mcal M\preceq\srED_\sigma\preceq\srED$. In particular all three relational systems share the same norms.
\end{lmm}
\begin{proof}
	Given $x\in\omom$, let $A_x=\set{y\in\omom}{ x\neqisig y}$ and note that $A_x=\Cup_{n\in\omega}A_x^n$, where 
	\begin{align*}
		A_x^n=\set{y\in\omom}{ \forall k> n\rb{x(k)\neq y(k)\text{ or }\sigma(x\restriction k)\neq\sigma(y\restriction k)} }.    
	\end{align*}
	We claim that each $A_x^n$ is nowhere dense. Let $s\in\fomom$ and assume $\dom(s)> n$. The suitability of $\sigma$ tells us that there exists $k> n$ and  $t\in{}^{k}\omega$ such that $s\subset t$ and $\sigma(t)=\sigma(x\restriction k)$. We define $t'=t^\frown{\ab{x(k)}}$, then $[t']\cap A^n_x=\emp$, showing that $A^n_x$ is nowhere dense. Consequently $A_x$ is meagre, hence we have the following Tukey connection $\ab{\omom,\mcal M,\in}\preceq\ab{\omom,\omom,\neqisig}$ given by $\rho_-:\omom\to\omom$ being the identity and $\rho_+:\omom\to\mcal M$ sending $x\mapsto A_x$. If $x\neqisig y$ holds, then $y\in A_x$ by definition.
	
	The other Tukey connection is obvious, since $x\neqi y$ implies that $x\neqisig y$.
\end{proof}

For our argument, it will be useful to work with a space $P=\prod_{n\in\omega}{}^{2n+2}\omega$. In order not to have to deal with double brackets, we will write $x_n=x(n)$ for $x\in P$. That is, if $x\in P$, then $x_n:2n+2\to\omega$ is a function for each $n\in\omega$, and $x:n\mapsto x_n$. We also define $P_{<\omega}=\Cup_{k\in\omega}\prod_{n\in k}{}^{2n+2}\omega$ to be the set of initial segments of elements of $P$, and similarly if $s\in P_{<\omega}$ and $n\in\dom(s)$, we write $s_n$ instead of $s(n)$.
Since ${}^{2n+2}\omega$ is countably infinite for each $n$, if we let $P$ have the product topology, then $P$ is homeomorphic to $\omom$.

\begin{thm}
	$\cov(\mcal M)\leq\ssn_c$ and dually $\ssn_c^\bot\leq\non(\mcal M)$.
\end{thm}
\begin{proof}
	Define $Q\subset P$ by $x\in Q$ if and only if there are infinitely many $n\in\omega$ such that $x_n$ is not the sequence $\ab{0,\dots,0}$ of length $2n+2$ containing only $0$'s. Note that $P\setminus Q$ is countable, hence $\norm{\ab{Q,\mcal M_P,\in}}=\cov(\mcal M)$ (where $\mcal M_P$ denotes the meagre ideal on $P$) and dually $\norm{\ab{\mcal M_P,Q,\not\ni}}=\non(\mcal M)$. It is furthermore easy to see that, for $\sigma:P_{<\omega}\to\omega$ suitable, $\ab{Q,\mcal M_P,\in}\preceq\ab{Q,Q,\neqisig}$; one follows the proof of \Cref{lmm:cm eds ed}.
	
	For the remainder of the proof, let us fix $\sigma:P_{<\omega}\to\omega$, defined by 
	\begin{align*}
		\sigma(s)=\sum\limits_{n\in\dom(s)}\ \sum\limits_{i\in\dom(s_n)}(s_n(i)+1),
	\end{align*} which is indeed suitable (with respect to $Q$). Here we use our assumption that if $x\in Q$, then there are infinitely many $n$ for which $x_n$ is not the sequence of only $0$'s. We will prove that $\ab{Q,Q,\neqisig}\preceq \ab{\fr S_{cc},\icoi,S_c}$.
	
	Given $y\in Q$, we have $y_n:2n+2\to\omega$. We define $\rho_-(y)=\mb a$ as follows. For any $n\in\omega$ and $k\in 2n+2$, let $j_k=\sum_{i<k}(y_n(i)+1)$, then we let \begin{align*}
		a_{i}&=\begin{cases}
			\frac{(-1)^k}{n}&\text{ if }i=\sigma(y\restriction n)+j_k\text{ for some $n\in\omega$ and }k\in2n+2,\\
			0&\text{ otherwise.}
		\end{cases}
	\end{align*}
	It is easy to see that $\mb a\in\fr S_{cc}$, that is, $\sum\mb a$ is conditionally convergent.
	
	To define $\rho_+:\icoi\to Q$, let $X\in\icoi$ and call $k\in\omega$ a \emph{switching point} of $X$ if $k=0$ or one of the following holds:
	\begin{itemize}
		\item $k\in X$ and $k-1\notin X$,
		\item $k\notin X$ and $k-1\in X$.
	\end{itemize}
	Let $\abset{i_l}{ l\in\omega}$ be the increasing enumeration of all switching points of $X$ and consider the intervals $I_l=[i_l,i_{l+1})$, where we note that $|I_l|>0$ for every $l\in\omega$. Let $N_n=n\cdot(n+1)$, then $N_{n+1}-N_n=2n+2$. We now define $x\in Q$ as follows: for every $n\in\omega$ and every $k\in 2n+2$, we define $x_n(k)=|I_{N_n+k}|-1$. We define $\rho_+(X)=x$. 
	
	Finally we show that this forms a Tukey connection. Let $X\in\icoi$ and $y\in Q$ and let $\mb a=\rho_-(y)$ and $x=\rho_+(X)$. Define $\chi_n=\sigma(x\restriction n)$ for each $n\in\omega$. Suppose that $n$ is such that $x_n=y_n$ and $\chi_n=\sigma(y\restriction n)$, then, using the notation $j_k$, $i_l$ and $N_n$ as in the constructions above, we see that for each $k\in2n+2$ we have $\chi_n+j_k=i_{N_n+k}$. Furthermore the only nonzero elements of $\mb a$ with index in $[\chi_n,\chi_{n+1})$ are $a_{\chi_n+j_k}$ for some $k\in 2n+2$. Since these indices are exactly the switching points of $X$ in the interval $[\chi_n,\chi_{n+1})$, we see that one of the following holds:
	\begin{itemize}
		\item $\chi_n+j_k\in X$ if and only if $k$ is even, or
		\item $\chi_n+j_k\in X$ if and only if $k$ is odd.
	\end{itemize}
	In both cases we see that $\card{\sum_{i\in X\cap [\chi_n,\chi_{n+1})}a_i}=1$. If we assume that $x\eqisig y$, then we can find arbitrarily large $n$ as above, hence it follows that there are infinitely many intervals $I$ such that $\card{\sum_{i\in X\cap I}a_i}=1$. This implies that $\sum_X\mb a$ is divergent.
\end{proof}

\subsection{Conditionally Convergent Subseries}\label{sec:cond conv series}

For the cardinal characteristic $\ssn_{cc}$ we can prove an additional lower bound and an upper bound. We will introduce a couple of well-known cardinal characteristics along the way. 

An infinite set $X\subset\omega$ is \emph{split by} a set $Y$ if $|X\cap Y|=|X\setminus Y|=\aleph_0$.  The \emph{splitting number} $\fr s$ is the least cardinality of a family $S\subset[\omega]^\omega$ such that every set $X\in[\omega]^\omega$ is split by some $Y\in S$. The \emph{reaping number} $\fr r$ is the least cardinality of a family $R\subset[\omega]^\omega$ such that no single $Y\in[\omega]^\omega$ splits all $X\in R$. Since no cofinite $Y\in[\omega]^\omega$ will split any $X\in[\omega]^\omega$, and since every cofinite $X\in[\omega]^\omega$ is split by any $Y\in\icoi$, we may without loss of generality replace $[\omega]^\omega$ by $\icoi$ in the definitions of $\fr s$ and $\fr r$.

\begin{thm}
	$\fr s\leq\ssn_{cc}$ and dually $\ssn_{cc}^\bot\leq\fr r$.
\end{thm}
\begin{proof}
	Let $\sr S=\ab{\icoi,\icoi,\ap{\text{is split by}}}$, then we will define a Tukey connection $\sr S\preceq \sr S_{cc}^{cc}$. We let $\rho_+:\icoi\to\icoi$ be the identity, and define $\rho_-:\icoi\to\fr S_{cc}$ as follows. Given a set $X\in\icoi$, we define an interval partition $\abset{I_n}{ n\in\omega}$ of $\omega$ such that 
	\begin{itemize}
		\item $\min(I_0)=0$,
		\item $\min(I_n)<\max(I_n)+1=\min (I_{n+1})$,
		\item $X=\Cup_{n\in\omega}I_{2n}$ or $X=\Cup_{n\in\omega}I_{2n+1}$ (depending on whether $0\in X$). 
	\end{itemize}
	For a given $i\in\omega$ we let $n\in\omega$ be such that $i\in I_n$ and define $a_i=\frac{(-1)^n}{|I_n|\cdot n}$. It is not hard to see that $\mb a\in \fr S_{cc}$. We let $\rho_-(X)=\mb a$.
	
	Let $X,Y\in\icoi$ and $\rho_-(X)=\mb a$. If $\sum_Y\mb a$ is conditionally convergent, then we claim that $X$ is split by $Y$. Suppose not, then either $X\cap Y$ or $X\setminus Y$ is finite. If $X\cap Y$ is finite and $0\in X$, then $a_i<0$ for only finitely many $i\in Y$. That means that $\sum_{Y\cap N_{\mb a}}\mb a$ is finite, hence $\sum_{Y}\mb a$ either tends to $\infty$ or is unconditionally convergent, contrary to our assumption that $\sum_Y\mb a$ is conditionally convergent. The cases where $X\setminus Y$ is finite or where $0\notin X$ are similar.
\end{proof}

We define two orderings $\leqs$ and $\leqi$ on $\omom$ by saying that $f\leqs g$ if and only if $f(n)>g(n)$ for at most finitely many $n\in\omega$ and $f\leqi g$ if and only if $f(n)\leq g(n)$ for infinitely many $n\in\omega$. If $f\leqs g$, we say that $g$ \emph{dominates} $f$. The \emph{dominating number} $\fr d$ is the least cardinality of a subset $D\subset\omom$ such that every $f\in\omom$ is dominated by some $g\in D$. Dually, the \emph{bounding number} $\fr b$ is the least cardinality of a subset $B\subset\omom$ such that no $g\in\omom$ dominates all $f\in B$. Note that such $B$ has the property that for every $g\in\omom$ there exists $f\in B$ such that $g\leqi f$. 

It is clear from this definition that $\sr D=\ab{\omom,\omom,\leqs}$ and $\sr B=\ab{\omom,\omom,\leqi}$ are relational systems with norms $\norm{\sr D}=\fr d$ and $\norm{\sr B}=\fr b$. It is also easy to see that $\sr D^\bot\eqv\sr B$.

Although the above systems are directly connected to $\fr b$ and $\fr d$ through their definitions, it will be useful to have an alternative (but Tukey-equivalent) relational system in terms of interval partitions. Let a sequence $\mcal I=\abset{I_n}{ n\in\omega}$ be called an \emph{interval partition} if each $I_n=[i_n,i_{n+1})$ is an interval of $\omega$ with $i_0=0$ and $i_n<i_{n+1}$ for all $n\in\omega$. We let $\mrm{IP}$ be the family of interval partitions. For $\mcal I,\mcal J\in\mrm{IP}$, we say that $\mcal J$ \emph{dominates} $\mcal I$, written $\mcal I\sqsubseteq^*\mcal J$, if for almost all $n\in\omega$ there is $k\in\omega$ such that $I_{k}\subset J_n$. Like the name of the relation suggests, $\sr D'=\ab{\mrm{IP},\mrm{IP},\sqsubseteq^*}$ is Tukey equivalent to $\sr D$ and a proof for this can be found in \cite[\S~2]{Blass}. We will also, in \Cref{sec:dual bonus}, consider a dual system $\sr B'=\ab{\mrm{IP},\mrm{IP},\sqsubseteq^\infty}$ where $\mcal I\sqsubseteq^\infty\mcal J$ if and only if there are infinitely many $n\in\omega$ for which there is $k\in\omega$ such that $I_k\subset J_n$. To show that $\sr B\eqv\sr B'$ and for the sake of completeness, we will repeat the proof from \cite[\S~2]{Blass} below with a slight adaptation.

\begin{lmm}
	$\sr D\eqv\sr D'$ and $\sr B\eqv\sr B'$.
\end{lmm}
\begin{proof}
	Given $g\in\omom$, we let $j_0=0$ and inductively define $j_{n+1}$ to be minimal such that $j_{n}<j_{n+1}$ and $g(x)<j_{n+1}$ for all $x\leq j_{n}$. We let $J_n=[j_n,j_{n+1})$ and define $\phi(g)=\mcal J=\abset{J_n}{ n\in\omega}\in\mrm{IP}$.
	
	On the other hand, given  $\mcal I=\abset{I_n}{ n\in\omega}\in\mrm{IP}$ with $I_n=[i_n,i_{n+1})$ for each $n\in\omega$, we define $f\in\omom$ by $f:x\mapsto i_{n+3}$, where $I_n$ is the interval in which $x$ is contained.\footnote{The proof in \cite[\S~2]{Blass} (in our notation) defines $f$ to map $x\mapsto i_{n+2}$.} We define $\psi(\mcal I)=f$.
	
	If we let $\rho_-=\phi$ and $\rho_+=\psi$, then this forms a Tukey connection for both $\sr D\preceq\sr D'$ and $\sr B\preceq\sr B'$. On the other hand, the choice of $\rho_-=\psi$ and $\rho_+=\phi$ forms a Tukey connection for $\sr D'\preceq\sr D$ and $\sr B'\preceq\sr B$. We only show that  $\sr B'\preceq\sr B$ and leave the rest of the details to the reader.
	
	Let $g\in\omom$ and $\mcal I\in\mrm{IP}$ and let $\mcal J=\phi(g)$ and $f=\psi(\mcal I)$. Suppose that $f\leqi g$ and $n_0\in\omega$. Find $n\geq n_0$ such that there is $x\in I_n$ with $f(x)\leq g(x)$, then we show there are $k',n'\in\omega$ such that $n'\geq n$ and $I_{n'}\subset J_{k'}$. Let $k\in\omega$ be such that $x\in J_k$. If $I_n\subset J_k$, then we are done, so let us assume that either $i_n<j_k$ or $j_{k+1}<i_{n+1}$.
	
	If $i_n<j_k$, we show that either $I_{n+1}\subset J_k$ or $I_{n+2}\subset J_{k+1}$. Assume that $I_{n+1}\nsubset J_k$, then
	\begin{align*}
		i_n<j_k\leq x<j_{k+1}\leq i_{n+2}<i_{n+3}=f(x)\leq g(x),   
	\end{align*} 
	and because $x<j_{k+1}$ implies $g(x)<j_{k+2}$, we see that $I_{n+2}\subset J_{k+1}$.
	
	Otherwise, if $j_{k+1}<i_{n+1}$, we have
	\begin{align*}
		x<j_{k+1}< i_{n+1}<i_{n+3}=f(x)\leq g(x)<j_{k+2}, 
	\end{align*}
	and thus $I_{n+1}\subset J_{k+1}$.
\end{proof}

The following result was presented to me by Jörg Brendle (private communication).

\begin{thm}
	$\ssn_{cc}\leq \fr d$ and $\fr b\leq\ssn_{cc}^\bot$.
\end{thm}
\begin{proof}
	We give a Tukey connection that witnesses $\srSS_{cc}^{cc}\preceq \sr D'$, consisting of $\rho_-:\fr S_{cc}\to\mrm{IP}$ and $\rho_+:\mrm{IP}\to\icoi$. 
	
	Given $\mcal J=\abset{J_n}{ n\in\omega}\in\mrm{IP}$, let $\rho_+(\mcal J)=\Cup_{n\in\omega}J_{2n}$. Given $\mb a\in\fr S_{cc}$, there is a function $f_{\mb a}:\bb R_{>0}\to\omega$ such that for every $\epsilon>0$ and interval $I=[n,n')$ with $n\geq f_{\mb a}(\epsilon)$ we have $\sum_I\mb a\in(-\epsilon,\epsilon)$. Define $\rho_-(\mb a)=\abset{I_n}{ n\in\omega}$ such that $\sum_{i\in I_n}|a_i|\geq 1$ and $\min(I_n)\geq f_{\mb a}\rb{\frac 1{n^2}}$ hold for each $n\in\omega$.
	
	If $\mb a\in\fr S_{cc}$ and $\mcal J\in\mrm{IP}$ with $\rho_-(\mb a)=\mcal I=\abset{I_n}{ n\in\omega}$ and $\rho_+(\mcal J)=X$ are such that $\mcal I\sqsubseteq^*\mcal J$, then for almost all $n\in\omega$ there is $k_n\in\omega$ such that $I_{k_n}\subset J_n$, and we can assume without loss that $k_n> n$. Particularly, $\min(J_n)\geq \min(I_n)\geq f_{\mb a}\rb{\frac 1{n^2}}$ holds for almost all $n\in\omega$, and thus $-\frac 1{n^2}<\sum_{J_{n}}\mb a<\frac 1{n^2}$ for almost all $n\in\omega$. This implies that $\sum_X\mb a$ is convergent. Moreover, since for almost all $n\in\omega$ there is $k$ with $I_{k}\subset J_{2n}$ and $\sum_{i\in I_{k}}|a_i|\geq 1$, we conclude that $\sum_{X}\mb a$ is not absolutely convergent.   
\end{proof}

\subsection{Unconditionally Convergent Subseries}

The final of our four subseries numbers related to convergence is the cardinal $\ssn_{ac}$. It turns out that this cardinal is equal to the dominating number. 

\begin{thm}\label{thm:ac and d}
	$\ssn_{ac}=\fr d$ and dually $\ssn_{ac}^\bot=\fr b$.
\end{thm}
\begin{proof}
	First, we give a Tukey connection witnessing $\sr S_{ac}^{cc}\preceq\sr D$, consisting of $\rho_-:\fr S_{cc}\to\omom$ and $\rho_+:\omom\to\icoi$.
	Given $\mb a\in\fr S_{cc}$, we let $\rho_-(\mb a)=g$ be any function such that $\ran(g)$ is coinfinite and $|a_{i}|\leq\frac1{n^2}$ for all $i\geq g(n)$. For $f\in\omom$, we let $\rho_+(f)=\ran(f)$ if $\ran(f)$ is coinfinite, and arbitrary otherwise. Let $\mb a\in\fr S_{cc}$ and $f\in\omom$, and let $g=\rho_-(\mb a)$ and $X=\rho_+(f)$. If $g\leqs f$ holds, then since $\ran(g)$ is coinfinite, so is $\ran(f)$. Now note that $|a_{f(n)}|\leq \frac 1{n^2}$ for almost all $n\in\omega$, and thus $\sum_X\mb a$ is unconditionally convergent.
	
	For the other direction we give a Tukey connection witnessing $\sr D\preceq\sr S_{ac}^{cc}$, consisting of $\rho_-:\omom\to\fr S_{cc}$ and $\rho_+:\icoi\to\omom$.
	Given $X\in\icoi$, let $\rho_+(X)=g:\omega\to X$ be the unique increasing bijection. Given $f\in\omom$, we will assume without loss of generality that $f$ is strictly increasing and $f(0)=0$. We define $\rho_-(f)=\mb a$ as follows. For every $i\in\omega$ there exists $n\in\omega$ such that $i\in[f(n),f(n+1))$, and we define $a_i=\frac{(-1)^i}{n}$. 
	
	Let $X\in\icoi$ and $f\in\omom$ and let $\rho_+(X)=g$ and $\rho_-(f)=\mb a$. To show this satisfies the conditions of a Tukey connection, we argue by contrapositive. Suppose that $f\nleqs g$, then $g\leqi f$ and thus we may find a sequence of integers $\abset{k_n}{ n\in\omega}$ such that $g(k_n)\leq f(k_n)$ and such that $k_{n+1}>2k_n$ for all $n\in\omega$. Now note that $|a_{g(i)}|> \frac1{k_{n+1}}$ for all $i\in[k_n,k_{n+1})$, hence 
	\begin{align*}
		\sum_{i\in[k_n,k_{n+1})}|a_{g(i)}|\geq (k_{n+1}-k_n)\cdot \frac1{k_{n+1}}=1-\frac{k_n}{k_{n+1}}\geq 1-\frac{k_n}{2k_n}=\frac12
	\end{align*}
	and therefore $\sum_{i\in\omega}|a_{g(i)}|=\infty$. Hence if $\mb a\restriction X$ is convergent, then it is conditionally so.
\end{proof}

\section{Subseries with Specific Limits}\label{sec:specific limits}

Let us define some additional subsets of $\fr S$. Given $A\subset\bb R_\infty$, we define
\begin{align*}
	\fr S_l^A&=\set{\mb a\in\fr S}{ \text{there is }r\in A\text{ such that }\ts\sum\mb a\text{ tends to }r}.
\end{align*}
Let us additionally define the following relations $S_l^A$ and $S_e$ on $\fr S\times\icoi$:
\begin{itemize}
	\item $(\mb a,X)\in S_l^A$ if and only if $\sum_X\mb a$ tends to a limit $r\in A$.
	\item $(\mb a,X)\in S_e$ if and only if $\sum\mb a$ and $\sum_X\mb a$ both tend to the same limit.
\end{itemize}

We may now define the subseries number $\ssn_l^A$ as the least size of a family $\mcal X\subset\icoi$ such that each conditionally convergent series $\sum\mb a$ has a subseries $\sum_X\mb a$, with $X\in\mcal X$, that converges to a limit in $A$. We define $\ssn_e$ in a similar way, except that $\sum_X\mb a$ must converges to the same limit as $\sum \mb a$ does. In terms of relational systems, $\ssn_l^A$ and $\ssn_e$ are the norms of the following systems.
\begin{align*}
	\srSS^{cc,A}_{l}&=\ab{\fr S_{cc},\icoi,S_l^A},\\
	\srSS^{cc}_{e}&=\ab{\fr S_{cc},\icoi,S_e}.
\end{align*}

In this section we will first show that the cardinals $\ssn_l^A$ and $\ssn_e$ are related, and that for $|A|<\fr c$ both cardinals are equal to $\fr c=2^{\aleph_0}$ and the dual cardinals are finite. By introducing an arbitrarily small margin for error to the definition of $\ssn_e$ and by considering $A$ that contain an open interval, we can create more interesting cardinals that turn out to be equal to $\ssn_c$, although the dual cardinals remain small and are equal to $\aleph_0$.

\subsection{Convergence to Specific Limits}

Let us first consider the case of singleton $A=\ac r$ for some $r\in\bb R$. We will need the following observation to relate $\ssn_e$ to $\ssn_l^{\ac r}$.

\begin{lmm}\label{lmm:sum decomposition}
	Let $\mb a\in\fr S_{cc}$ and $X\in\icoi$ be such that $\sum_X\mb a$ is convergent, and let and $r,r_X\in\bb R$ be the limits $\sum\mb a=r$ and $\sum_X\mb a=r_X$. Then $\sum_{\omega\setminus X}\mb a=r-r_X$.
\end{lmm}
\begin{proof}
	This is easy by considering the decomposition $\sum_k\mb a=\sum_{k\cap X}\mb a+\sum_{k\cap (\omega\setminus X)}\mb a$ of the partial sums for all $k\in\omega$.
\end{proof}

\begin{crl}\label{lmm:ss_e and 0}
	$\ssn_e=\ssn_l^{\ac 0}$ and $\ssn_e^\bot={\ssn_l^{\ac 0\bot}}$.
\end{crl}
\begin{proof}
	We will give a Tukey connection $\sr S_e^{cc}\preceq\sr S_l^{cc,\ac 0}$ consisting of $\rho_-:\fr S_{cc}\to\fr S_{cc}$ being the identity function and $\rho_+:\icoi\to\icoi$ mapping $X\mapsto\omega\setminus X$. \Cref{lmm:sum decomposition} shows us that $\sum_X\mb a=\sum\mb a$ if and only if $\sum_{\omega\setminus X}\mb a=0$. 
	
	Note that the same pair $(\rho_-,\rho_+)$ is also a Tukey connection for $\sr S_l^{cc,\ac 0}\preceq\sr S_e^{cc}$.
\end{proof}

\begin{thm}
	$\ssn_e=\ssn_l^A=\fr c$  for all $A\in[\bb R]^{<\fr c}$ and $\ssn_e^\bot={\ssn_l^{B\bot}}=2$ for all $B\in[\bb R]^{<\omega}$.
\end{thm}
\begin{proof}
	Let $\mb a\in\fr S_{cc}$ be arbitrary and for each $t\in\bb R$ let $\mb a^t$ be given by $a^t_i=a_i+\frac t{i^2}$ It is clear that $\mb a^t\in\fr S_{cc}$, since $\sum_{i\in\omega}\frac t{i^2}$ is absolutely convergent, and the termwise sum of a conditionally convergent series with an absolutely convergent series is conditionally convergent. 
	
	Now suppose that $X\in\icoi$ is such that $\sum_X\mb a^t=r$ and $s\in\bb R$ is such that $t\neq s$. Then $\sum_X\mb a^{s}\neq\sum_X\mb a^t$. Therefore $\set{t\in\bb R}{ \exists X\in\mcal X\rb{\sum_X\mb a^t\in A}}$ has cardinality $|\mcal X|\cdot |A|$, and thus it follows that $\ssn_l^A=\fr c$ when $|A|<\fr c$. By \Cref{lmm:ss_e and 0} we also have $\ssn_e=\ssn_l^{\ac 0}=\fr c$.
	
	For the dual cardinal ${\ssn_l^{B\bot}}$, let $\epsilon>0$ be such that $|x-y|\geq\frac{\epsilon\cdot \pi^2}6$ for all $x,y\in B$.  Suppose $\mb a\in\fr S_{cc}$ and let $\mb b$ have $b_i=a_i+\frac{\epsilon}{i^2}$, then $\sum_X\mb b- \sum_X\mb a<\frac{\epsilon\cdot \pi^2}6$ for any $X\in\icoi$. Hence, ${\ssn_l^{B\bot}}=2$ for every finite $B\subset\bb R$ and, by \Cref{lmm:ss_e and 0}, $\ssn_e^\bot={\ssn_l^{\ac 0\bot}}$.
\end{proof}
We note that the above proof shows that ${\ssn_l^{B\bot}}=2$ even for countable $B\subset\bb R$ for which $\inf\set{|x-y|}{x,y\in B\la x\neq y}>0$ holds.

\subsection{Introducing Error Margins}

Let $\epsilon>0$, then we define $\ssn_e^\epsilon$ as the norm of the system $\sr S_e^{cc,\epsilon}=\ab{\fr S_{cc},\icoi,S_e^\epsilon}$ where $(\mb a,X)\in S_e^\epsilon$ if and only if $\sum_X\mb a$ is convergent with a limit in the open interval of length $2\epsilon$ centred around the limit of $\sum\mb a$.

\begin{lmm}\label{epsilon equality}
	Let $\epsilon>0$, then $\ssn_e^\epsilon=\ssn_l^{(-\epsilon,\epsilon)}$ and ${\ssn_e^{\epsilon\bot}}={\ssn_l^{(-\epsilon,\epsilon)\bot}}$.
\end{lmm}
\begin{proof}
	We will show that $\sr S_e^{cc,\epsilon}\eqv\sr S_l^{cc,(-\epsilon,\epsilon)}$. Consider the Tukey connection consisting of $\rho_-:\fr S_{cc}\to\fr S_{cc}$ being the identity function and $\rho_+:\icoi\to\icoi$ mapping $X\mapsto\omega\setminus X$. 
	Suppose $(\mb a,X)\in S_e^\epsilon$, then there is $\delta\in(-\epsilon,\epsilon)$ such that $\delta+\sum_X\mb a=\sum\mb a$. By \Cref{lmm:sum decomposition} we have then that $\sum_{\omega\setminus X}\mb a=-\delta$, and thus $(\mb a,\omega\setminus X)\in S_l^{(-\epsilon,\epsilon)}$. This shows $\sr S_e^{cc,\epsilon}\preceq\sr S_l^{cc,(-\epsilon,\epsilon)}$, but note that the same pair $(\rho_-,\rho_+)$ is also a Tukey connection for $\sr S_l^{cc,(-\epsilon,\epsilon)}\preceq\sr S_e^{cc,\epsilon}$.
\end{proof}

\begin{lmm}
	Let $A,B\subset\bb R$ be bounded open intervals, then $\ssn_l^A=\ssn_l^B$.
\end{lmm}
\begin{proof}
	We will show that $\ssn_l^A\leq\ssn_l^B$ holds in the following two cases:
	\begin{itemize}
		\item (Scaling) If $A=(-1,1)$ and $B=(-\epsilon,\epsilon)$ for some $\epsilon>0$.
		\item (Translation) If $A=(c,c')$ and $B=(b+c,b+c')$ for some $b\in\bb R$.
	\end{itemize}
	This suffices to prove the lemma, as scaling and translation suffice to transform any bounded open interval into any other bounded open interval.
	
	For scaling, we prove $\sr S_l^{cc,(-1,1)}\eqv\sr S_l^{cc,(-\epsilon,\epsilon)}$. Let $\rho_-:\fr S_{cc}\to\fr S_{cc}$ map $\mb a=\abset{a_i}{ i\in\omega}$ to $\mb a^\epsilon=\abset{\epsilon\cdot a_i}{ i\in\omega}$ and let $\rho_+:\icoi\to\icoi$ be the identity. It is clear that  $\sum_X\mb a^\epsilon=\epsilon\cdot\sum_X\mb a$, and thus $\sum_X\mb a^\epsilon\in (-\epsilon,\epsilon)$ if and only if $\sum_X\mb a\in(-1,1)$.
	
	For translation, let $\mcal X$ be a witness for $\ssn_l^{(c,c')}$ and assume without loss of generality that $x\mathbin\triangle X\in\mcal X$ for every $x\in[\omega]^{<\omega}$ and $X\in\mcal X$. Let $\mb a\in\fr S_{cc}$ and find $X\in\mcal X$ such that $\sum_X\mb a\in(c,c')$. Say that $r=\sum_X\mb a$ is the limit value, then we may find some $\epsilon>0$ such that $(r-\epsilon,r+\epsilon)\subset (c,c')$. 
	
	Let $\mb a'$ be defined by $a_i'=a_i$ if $i\notin X$ and $a_i'=-a_i$ otherwise. Since we only change the signs of terms in $X$ and $\sum_X\mb a$ is convergent, it follows that $\sum_{P_{\mb a'}}\mb a'=\infty$ and $\sum_{N_{\mb a'}}\mb a'=-\infty$. Therefore we can find some finite $x\in[\omega]^{<\omega}$ such that $\sum_x\mb a'\in(b-\epsilon,b+\epsilon)$. Note that $\sum_{x\triangle X}\mb a=(\sum_x\mb a')+(\sum_X\mb a)$, because for all $i\in x\triangle X$ we have $a'_i=a_i$ and for all $i\in x\cap X$ we have $a'_i=-a_i$, and thus the latter terms cancel out by $x$ being finite. It follows that 
	\begin{align*}
		\ts\sum_{x\triangle X}\mb a\in (b+r-\epsilon, b+r+\epsilon)\subset (b+c,b+c').
	\end{align*}
	Therefore, since by assumption $x\mathbin\triangle X\in\mcal X$, we see that $\mcal X$ is also a witness for $\ssn_l^B$.
\end{proof}

\begin{thm}
	$\ssn_e^\epsilon=\ssn_l^A=\ssn_c$ for any $\epsilon>0$ and any $A\subset\bb R$ such that $A$ contains a nonempty open interval.
\end{thm}
\begin{proof}
	Firstly, $\ssn_c\leq\ssn_e^\epsilon$ is obvious by \Cref{trivial inference}.
	
	Let first take $A=(-1,1)$ and show that $\ssn_l^A\leq\ssn_c$. Let $\mcal X\subset\icoi$ witness $\ssn_c$, and assume that $X\setminus k\in\mcal X$ for all $k\in\omega$. Let $\mb a\in\fr S_{cc}$ and $X\in\mcal X$ be such that $\sum_X\mb a=r$ is convergent. Then there is some $k\in\omega$ such that $\sum_{X\cap k}\mb a\in (r-1,r+1)$ and such that the partial sum $|\sum_{n\cap X\setminus k}\mb a|<1$ for all $n\geq k$. It then follows that $\sum_{X\setminus k}\mb a\in(-1,1)$. Hence $\mcal X$ is a witness for $\ssn_l^{(-1,1)}$ as well. 
	
	Unfixing $A$, the previous two Lemmas imply that $\ssn_e^\epsilon = \ssn_l^A=\ssn_c$ for all $\epsilon>0$ and bounded open intervals $A\subset\bb R$. Moreover, it is obvious that $\ssn_l^\bb R=\ssn_c$ by definition. If $B\subset A$, then we easily see that $\ssn_l^A\leq\ssn_l^B$. Hence, if $B$ is a (bounded) nonempty open interval, then $\ssn_c=\ssn_l^\bb R\leq\ssn_l^A\leq\ssn_l^B=\ssn_c$.
\end{proof}

Interestingly we have no duality:

\begin{thm}
	${\ssn_e^{\epsilon\bot}}=\aleph_0$ for every $\epsilon>0$.
\end{thm}
\begin{proof}
	For convenience we will take $\epsilon=1$. We already know that ${\ssn_e^{\epsilon\bot}}={\ssn_l^{(-1,1)\bot}}$ from \Cref{epsilon equality}. We show that the latter is equal to $\aleph_0$.
	
	Let $\mb a=\abset{a_i}{ i\in\omega}\in\fr S_{cc}$ and define $\mb a^n$ by $a_i^n=a_i+\frac{n}{i^2}$. Suppose that $X\in\icoi$ is such that $\sum_X\mb a=r\in(-1,1)$ and let $s=\sum_X\frac{1}{i^2}$, which as a subseries of an absolutely convergent series is itself absolutely convergent and clearly nonzero. Hence $\sum_X\mb a^n=r+n\cdot s$. Therefore we can choose $n$ large enough such that $\sum_X\mb a^n\notin(-1,1)$. This shows that ${\ssn_l^{(-1,1)\bot}}\leq\aleph_0$.
	
	In the other direction, suppose $\mb a^0,\dots,\mb a^k\in\fr S_{cc}$, then we find some $X\in\icoi$ such that $\sum_X\mb a^i$ converges for every $0\leq i\leq k$ (using that $\ssn_c^\bot$ is infinite). There must be some $n_i$ such that $\sum_{X\setminus n}\mb a^i\in(-1,1)$ for all $n\geq n_i$. Let $N=\max\set{n_i}{ 0\leq i\leq k}$, then $\sum_{X\setminus N}\mb a^i\in(-1,1)$ for all $0\leq i\leq k$.
\end{proof}

We remark that the first part of the above proof also shows that ${\ssn_l^{A\bot}}\leq\aleph_0$ for any $A\subset\bb R$ that is not dense: for each $q\in\bb Q$ define $\mb a^q$ by $a_i^q=a_i+\frac q{i^2}$, then $\set{\sum_X\mb a^q}{q\in\bb Q}$ is dense for any $X\in\icoi$.

\section{Conditional and Unconditional Series}\label{sec:unconditional series}

Remember that we call $\mb a\in\fr S$ \emph{conditional} if both $\sum_{P_{\mb a}}\mb a$ and $\sum_{N_{\mb a}}\mb a$ tend to infinity. Every series that diverges by oscillation is easily seen to be conditional, but not every convergent series is, as convergent series could be absolutely convergent. A similar thing can be said for series that tend to infinity. Let us say that $\sum\mb a$ tends conditionally to infinity if $\sum\mb a$ tends to infinity and $\mb a$ is conditional, and otherwise we will say that $\sum\mb a$ tends unconditionally to infinity. We introduce the following subsets $\fr S$.
\begin{align*}
	\fr S_n&=\set{\mb a\in\fr S}{ \ts\sum\mb a\text{ is conditional}},\\
	\fr S_{ci}&=\set{\mb a\in\fr S}{ \ts\sum\mb a\text{ tends conditionally to infinity}},\\
	\fr S_{ui}&=\set{\mb a\in\fr S}{ \ts\sum\mb a\text{ tends unconditionally to infinity}}.
\end{align*}

Under our naming convention this provides us with new cardinal characteristics, that we will discuss in the next subsections.

\subsection{Subseries that are Conditional}
Let us consider $\ssn_n$, which is the least cardinality of a set $\mcal X\subset\icoi$ such that for each conditionally convergent $\mb a$ there is some $X\in\mcal X$ such that $\sum_X\mb a$ is still a conditional series. In other words, $\ssn_n$ is the norm of $\srSS^{cc}_n$.

For a rather pathological reason we can see that $\ssn_n$ is finite.

\begin{lmm}\label{finite ssn}
	$\ssn_n=3$
\end{lmm}
\begin{proof}
	Firstly $\ssn_n>2$. Namely, if we have $\mcal X=\ac{X_0,X_1}\subset\icoi$, then either:
	\begin{itemize}
		\item $Z=\omega\setminus(X_0\cup X_1)$ is infinite, in which case we define $\mb a\in\fr S_{cc}$ such that $P_{\mb a}\subset Z$, then clearly $\mb a\restriction X_i$ contains no positive terms for either $i$, or
		\item $Z_0=X_0\setminus X_1$ and $Z_1=X_1\setminus X_0$ are both infinite, in which case we define $\mb a\in\fr S_{cc}$ such that $P_{\mb a}\subset Z_0$ and $N_{\mb a}\subset Z_1$, then clearly $\mb a\restriction X_0$ contains no negative terms and $\mb a\restriction X_1$ contains no positive terms.
	\end{itemize}
	
	Next we show that the following three sets suffice: let $X_i=\omega\setminus\set{3k+i}{ k\in\omega}$ for each $i\in 3$. Let $\mb a\in\fr S_{n}$, then we consider for each $i\in 3$: 
	\begin{align*}
		P_i=P_{\mb a}\cap\set{3k+i}{ k\in\omega},\\
		N_i=N_{\mb a}\cap\set{3k+i}{ k\in\omega}.
	\end{align*}
	Now we can choose $i_P,i_N\in 3$ such that $\sum_{P_{i_P}}\mb a=\infty$ and $\sum_{N_{i_N}}\mb a=-\infty$. But note that if $j\in 3\setminus\ac{i_P,i_N}$, then $P_{i_P},N_{i_N}\subset X_j$, hence $\sum_{X_j}\mb a$ is conditional.
\end{proof}

\subsection{Unconditional Series as the Set of Challenges}

So far, the subseries numbers have been defined as norms of relational systems $\ab{\fr S_{cc},\icoi,R}$ for some relation $R$. In particular, our set of challenges has always been $\fr S_{cc}$. In this section we briefly consider what happens if we allow for unconditional challenges, specifically the sets $\fr S_{ac}$ and $\fr S_{ui}$.

First, let us remark that if $\mb a\in\fr S_{ac,ui}$, that is, if $\sum\mb a$ is not a conditional series, then $\sum_X\mb a$ is not conditional for any $X\in\icoi$ either. Moreover, subseries of absolutely convergent series are absolutely convergent, thus if $\mb a\in\fr S_{ac}$, then $\mb a\restriction X\in\fr S_{ac}$ as well. This makes $\ssn_{ac}^{ac}=1$ and $\ssn_{ac}^{ac\bot}$ undefined.

It is, however possible for a divergent unconditional series (which necessarily tends to infinity) to have a convergent subseries. The subseries number $\ssn_c^{ui}$ is the least cardinality of a subset $\mcal X\subset \icoi$ such that every series that tends unconditionally to an infinite value has a subseries given by some $X\in\mcal X$ that is convergent. Note that the previous paragraph implies $\ssn^{ui}_c=\ssn^{ui}_{ac}$. It turns out that $\ssn_c^{ui}$ is another characterisation of the dominating number. We omit the proof, as it is almost entirely similar to the proof of \Cref{thm:ac and d}.

\begin{thm}
	$\ssn_c^{ui}=\fr d$ and ${\ssn_c^{ui\bot}}=\fr b$
\end{thm}

\section{Bonus: Dualising the Original Subseries Numbers}\label{sec:dual bonus}

In the original subseries paper \cite{Subseries}, several relations between the cardinal characteristics $\ssn_i$, $\ssn_o$, $\ssn_{i,o}$ and various other cardinal characteristics are shown. Although the article does not use the structure of relational systems that we have used in our article, it is not hard to see that almost every proof from the original paper can be written in the form of a Tukey connection. Specifically, the proofs that $\cov(\mcal N)\leq\ssn_{i,o}$, that $\fr s\leq\ssn_{i,o}$, that $\cov(\mcal M)\leq\ssn_i$ and that $\ssn_o\leq\non(\mcal M)$ all can be written using a Tukey connection. Even the theorems connecting the subseries numbers with the rearrangement numbers by $\fr{rr}\leq\max\ac{\fr b,\ssn_{i,o}}$ and $\fr{rr}_i\leq\max\ac{\fr d,\ssn_i}$, can be readily translated into a Tukey connection using a sequential composition (see below).

We therefore observe the following facts about the dual subseries numbers:
\begin{align*}
	&\ssn_{i,o}^\bot\leq \non(\mcal N)&&\ssn_{i,o}^\bot\leq\fr r\\
	&\ssn_i^\bot\leq\non(\mcal M)&&\cov(\mcal M)\leq\ssn_o^\bot\\
	&\min\ac{\fr d,\ssn_{i,o}^\bot}\leq \fr{rr}&&\min\ac{\fr b,\ssn_i^\bot}\leq\fr{rr}_i^\bot
\end{align*}
Some of these statements are made trivial by the fact that $\ssn_i^\bot=4$, as was shown in \cite{Brian2018}. Although it is not known whether $\ssn_i<\fr c$ is consistent, the fact that the dual cardinal $\ssn_i^\bot$ is finite is a strong indicator that $\ssn_i=\fr c$ might be provable in $\msf{ZFC}$.

It is unknown whether $\ssn_o=\ssn_{i,o}$ is provable, nor do we know this about their duals. Nevertheless,  \cite[\S~8]{Subseries} shows a weaker claim, that $\fr b\leq\ssn_{i,o}$ implies $\ssn_o=\ssn_{i,o}$, or succinctly that $\ssn_o\leq\max\ac{\fr b,\ssn_{i,o}}$. The original proof does not yield a Tukey connection, because the cases of divergence by oscillation and tending to infinity are treated separately. 

In this section, we will give a variation of the proof that allows us to state it in terms of Tukey connections, and thus provides a dual result. In order to do this, we must also consider the splitting number $\fr s$, but this will have no effect on the result. We need the following lemma.

\begin{lmm}\label{szo sz lemma}
	Let $\mb a\in\fr S_{cc}$ and $X\in\icoi$. If $\sum_{X}\mb a$ diverges, there is a positive real number $c\in\bb R_{>0}$ for which there are infinitely many disjoint intervals $I_i\subset\omega$ such that one of the following is true:
	\begin{itemize}
		\item $\sum_{n\in I_i\setminus X}a_n<-c$ and $c<\sum_{n\in I_i\cap X}a_n$ for all $i\in\omega$, or
		\item $\sum_{n\in I_i\cap X}a_n<-c$ and $c<\sum_{n\in I_i\setminus X}a_n$ for all $i\in\omega$.\qedhere
	\end{itemize}
\end{lmm}
\begin{proof}
	We know that $\sum_{n\in\omega} a_n$ converges, thus for any $\epsilon>0$ we could find some $m'\in\omega$ for which $\sum_{m'\leq n\leq m}a_n\in(-\epsilon,\epsilon)$ for every $m>m'$. 
	
	If $\sum_{X}\mb a$ diverges by oscillation, find $d\in\bb R$ and $c\in{\bb R}_{>0}$ such that $\sum_{X\cap m}\mb a<d-c$ for infinitely many $m\in \omega$ and $\sum_{X\cap m}\mb a > d+c$ for infinitely many $m\in\omega$. Fix $\epsilon<c$ and $m'\in\omega$ as above, and also fix $m_i,k_i\in\omega$ for each $i\in\omega$ such that for all $i\in\omega$:
	\begin{itemize}
		\item $m'<m_0$,
		\item $m_i<k_i<m_{i+1}$, 
		\item $\sum_{X\cap m_i}\mb a<d-c$, and 
		\item $\sum_{X\cap k_i}\mb a>d+c$. 
	\end{itemize}
	It follows from this that $\sum_{X\cap[m_i,k_i)}\mb a>2c$ and $\sum_{[m_i,k_i)}\mb a<\epsilon<c$ (since $m_i>m'$), and thus $\sum_{ [m_i,k_i)\setminus X}\mb a<-c$. We can therefore take $I_i=[m_i,k_i)$ for each $i\in\omega$ to see that the first bullet point from the lemma holds.
	
	If $\sum_{X}\mb a=\infty$, then we could take any $c\in\bb R_{>0}$ and fix some $\epsilon<c$ and $m'$ as above. We could fix $m_i$ for all $i\in\omega$ such that:
	\begin{itemize}
		\item $m'<m_0$,
		\item $m_i<m_{i+1}$, 
		\item $\sum_{X\cap [m_i,m_{i+1})}\mb a>2c$. 
	\end{itemize}
	For the same reasons as before, we see that $\sum_{[m_i,m_{i+1})\setminus X}\mb a<c$. Therefore we take $I_i=[m_i,m_{i+1})$ for each $i\in\omega$ to see that the first bullet point from the lemma holds.
	
	Finally if $\sum_{X}\mb a=-\infty$, we have the same argument as above (with $\sum_X\mb a=\infty$), except we now let $\sum_{n\in X\cap [m_i,m_{i+1})}a_n < -2c$ and get $\sum_{n\in [m_i,m_{i+1})\setminus X}a_n>c$, meaning that the second bullet point from the lemma holds instead.
\end{proof}

To show the theorem, we will use a Tukey connection between the relational system $\srSS^{cc}_{i,o}$ and a composition of the three systems $\srSS^{cc}_o$, $\sr B$ and $\sr S$. For this we define the \emph{sequential composition} of two relational systems $\sr X=\ab{X,Y,R}$ and $\sr X'=\ab{X',Y',R'}$ as the relational system $\sr X^\frown\sr X'=\ab{X\times {}^YX',Y\times Y',K}$ where $((x,f),(y,y'))\in K$ if and only if $(x,y)\in R$ and $(f(y),y')\in R'$. If $\sr X$ and $\sr X'$ have infinite norms, then $\norm{\sr X{}^\frown\sr X'}=\max\ac{\norm{\sr X},\norm{\sr X'}}$ and $\norm{(\sr X{}^\frown\sr X')^\bot}=\min\ac{\norm{\sr X^\bot},\norm{\sr X'{}^\bot}}$, see also \cite[Thm 4.11]{Blass}.

\begin{thm}\label{szo sz proof}
	$\ssn_o\leq\mrm{max}\ac{\fr{b},\ssn_{i,o}}$ and $\mrm{min}\ac{\fr{d},\ssn_{i,o}^\bot}\leq \ssn_o^\bot$.
\end{thm}
\begin{proof}
	Remember the systems $\sr S=\ab{\icoi,\icoi,\ap{\text{is split by}}}$ and $\sr B'=\ab{\mrm{IP},\mrm{IP},\sqsubseteq^\infty}$, that we introduced in \Cref{sec:cond conv series}. We will give a Tukey connection $\srSS^{cc}_o\preceq(\srSS^{cc}_{i,o}{}^\frown \sr B'){}^\frown\sr S$. We will explicitly write out the relational system for $(\srSS^{cc}_{io}{}^\frown \sr B'){}^\frown\sr S$ below:
	\begin{align*}
		\ab{\fr S_{cc}\times {}^{\icoi} \mrm{IP}\times{}^{\icoi\times\mrm{IP}}\icoi,\quad
			\icoi\times\mrm{IP}\times\icoi,\quad 
			Z}
	\end{align*}
	Here $Z$ is the defined for $\mb a\in\fr S_{cc}$, functions $f:{\icoi}\to{\mrm{IP}}$ and $g:{\icoi}\times{\mrm{IP}}\to{\icoi}$, interval partition $\mcal I=\abset{I_k}{ k\in\omega}$ and sets $X,S\in\icoi$, by $((\mb a,f,g),(X,\mcal I,S))\in Z$ if and only if all of the following hold:
	\begin{itemize}
		\item $\sum_X\mb a$ diverges,
		\item for interval partition $f(X)=\abset{J_n}{ n\in\omega}$, there are infinitely many $k\in\omega$ such that $J_n\subset I_k$ for some $n\in\omega$, and
		\item $g(X,\mcal I)$ is split by $S$.
	\end{itemize}
	The Tukey connection is witnessed by functions $\rho_-:\fr S_{cc}\to\fr S_{cc}\times {}^{\icoi} \mrm{IP}\times{}^{\icoi\times\mrm{IP}}\icoi$ and $\rho_+:\icoi\times\mrm{IP}\times\icoi\to\icoi$.
	
	Of these two functions $\rho_+$ is the easier one: let $X,S\in\icoi$ and $\mcal I=\abset{I_k}{ k\in\omega}\in\mrm{IP}$, then we define $\rho_+$ as:
	\begin{align*}
		\rho_+(X,\mcal I,S)=\rb{\ts\Cup_{k\in S}I_k\cap X}\cup\rb{\ts\Cup_{k\notin S}I_k\setminus X}.
	\end{align*}
	
	To define $\rho_-$, we let $\rho_-(\mb a)=(\mb a,f,g)$, where $f$ and $g$ are defined as below.
	
	Let $X\in[\omega]^\omega_\omega$ and $\mcal I=\abset{I_k}{ k\in\omega}\in\mrm{IP}$. If $\sum_{X}\mb a$ converges, we take any arbitrary value for $f(X)$ and $g(X,\mcal I)$, as this case will be irrelevant. Therefore, assume that $\sum_{X}\mb a$ diverges. By \Cref{szo sz lemma} we can find a family of disjoint intervals $\set{K_n}{ n\in\omega}$ and a positive real number $c\in\bb R_{>0}$ such that $\sum_{K_n\setminus X}\mb a<-c$ and $c<\sum_{K_n\cap X}\mb a$ for all $n\in\omega$ or such that $\sum_{K_n\cap X}\mb a<-c$ and $c<\sum_{K_n\setminus X}\mb a$ for all $n\in\omega$. We will assume without loss of generality that the first is the case and we will also assume that $\mrm{min}(K_n)<\mrm{min}(K_{n+1})$ for all $n\in\omega$. We let $f(X)=\abset{J_n}{ n\in\omega}$ where $J_n=[\min(K_n),\min(K_{n+1}))$.
	
	In case there are only finitely many $k\in\omega$ such that $J_n\subset I_k$ for some $n\in\omega$, we let $g(A,B)$ be arbitrary, as this case will once again be irrelevant. Otherwise we will define $g(X,\mcal I)=\set{k\in\omega}{ \exists n\in\omega(J_n\subset I_k)}$.
	
	Now we show that this is indeed a Tukey connection. 
	Let $C=\rho_+(X,\mcal I,S)$. Suppose that $((\mb a,f,g),(X,\mcal I,S))\in Z$, then $\sum_X\mb a$ diverges, thus $f(X)=\abset{J_n}{ n\in\omega}$ is defined non-arbitrarily as above, and there are infinitely many $k\in\omega$ such that $J_n\subset I_k$ for some $n\in\omega$, thus also $g(X,\mcal I)=Y$ is defined non-arbitrarily. Finally we have that $Y$ is split by $S$, thus $Y\cap S$ and $Y\setminus S$ are both infinite. 
	
	If $k\in Y\cap S$, then by definition of $\rho_+$ we have $I_k\cap X=I_k\cap C$. Furthermore, by definition of $g$ there is some $n\in\omega$ such that $J_n\subset I_k$, and since we defined $J_n=[\min(K_n),\min(K_{n+1}))$ and the intervals $K_n$ and $K_{n+1}$ are disjoint, it follows that $K_n\cap X=K_n\cap C$. We chose $K_n$ such that $\sum_{K_n\cap X}\mb a>c$, hence we see that $\sum_{K_n\cap C}\mb a>c$

	On the other hand, if $k\in Y\setminus S$, then $I_k\setminus X=I_k\cap C$. By the same argument as above, we find some $n\in\omega$ such that $\sum_{K_n\setminus X}\mb a=\sum_{K_n\setminus C}\mb a<-c$.
	
	We conclude that $\sum_C\mb a$ diverges by oscillation. Therefore, our Tukey connection has shown that $\ssn_o\leq \norm{(\srSS^{cc}_{i,o}{}^\frown {\sr B}'){}^\frown\sr S}=\mrm{max}\ac{\ssn_{i,o},\fr b,\fr s}=\mrm{max}\ac{\ssn_{i,o},\fr b}$, where the last equality follows from $\fr s\leq\ssn_{i,o}$. Dually we get $\ssn_o^\bot\geq\mrm{min}\ac{\ssn_{i,o}^\bot,\fr d}$.
\end{proof}

\begin{crl}
	If $\fr b\leq \ssn_{i,o}$, then $\ssn_{i,o}=\ssn_o$. Dually, if $\fr d\geq\ssn_{i,o}^\bot$, then $\ssn_{i,o}^\bot=\ssn_o^\bot$.
\end{crl}

\section{Conclusion and Open Questions}

We may summarise the results from this article in \Cref{fig:diagram}, showing the relative order of the cardinal characteristics mentioned so far compared to the cardinals of the Cicho\'n diagram and $\fr s$ and $\fr r$. In the diagram, we use dotted lines for the relations between classical cardinal characteristics, such as $\fr s$, $\fr d$ or $\cov(\mcal M)$, and dashed arrows are used where it is unknown whether the ordering between the cardinals is consistently strict.

\begin{figure}[ht]\centering
	\begin{tikzpicture}[yscale=1.2,xscale=1.6]
		\node[] (co) at (8.4,8) {$2^{\aleph_0}$};
		\node[] (aN) at (1.2,1) {$\add(\mcal N)$};
		\node[] (p) at (1.2,0) {$\aleph_1$};
		\node[] (s) at (2.2,3) {$\fr s$};
		\node[] (r) at (7.4,5) {$\fr r$};
		\node[] (io) at (2.2,7) {$\ssn_{i,o}$};
		\node[] (io') at (7.4,1) {$\ssn_{i,o}^\bot$};
		\node[] (cN) at (1.2,7) {$\cov(\mcal N)$};
		\node[] (o) at (3,7) {$\ssn_o$};
		\node[] (o') at (6.6,1) {$\ssn_o^\bot$};
		\node[] (i) at (7,8) {$\ssn_i$};
		\node[] (cM) at (5.6,1) {$\cov(\mcal M)$};
		\node[] (c) at (5.7,2) {$\ssn_c$};
		\node[] (c') at (3.9,6) {$\ssn_c^\bot$};
		\node[] (d) at (5.9,4) {$\fr d=\ssn_{ac}=\ssn_{c}^{ui}$};
		\node[] (cc) at (5.8,3) {$\ssn_{cc}$};
		\node[] (cc') at (3.8,5) {$\ssn_{cc}^\bot$};
		\node[] (b) at (3.7,4) {$\fr b=\ssn_{ac}^\bot=\ssn_c^{ui\bot}$};
		\node[] (nM) at (4,7) {$\non(\mcal M)$};
		\node[] (aM) at (3.4,1) {$\add(\mcal M)$};
		\node[] (fM) at (6.2,7) {$\cof(\mcal M)$};
		\node[] (nN) at (8.4,1) {$\non(\mcal N)$};
		\node[] (fN) at (8.4,7) {$\cof(\mcal N)$};

		\draw[->,dotted] (p) -- (aN);
		\draw[->,dotted] (aN) -- (aM);
		\draw[->,dotted] (aM) -- (cM);
		\draw[->,dotted] (aM) -- (b);
		\draw[->,dotted] (nM) -- (fM);
		\draw[->,dotted] (d) -- (fM);
		\draw[->,dotted] (aN) -- (cN);
		\draw[->,dotted] (fM) -- (fN);
		\draw[->,dotted] (nN) -- (fN);
		\draw[->,dashed] (cc) -- (d);
		\draw[->] (s) -- (io);
		\draw[->] (s) -- (cc);
		\draw (s) edge[->,dotted, out=0, in=150] (nN);
		\draw (cN) edge[->,dotted, out=-30, in=180] (r);
		\draw (p) edge[->,dotted,out=45,in=270] (s);
		\draw (r) edge[->,dotted,out=90,in=225] (co);
		\draw[->,dotted] (b) -- (d);
		\draw[->,dashed] (b) -- (cc');
		\draw[->,dashed] (cc') -- (c');
		\draw[->,dashed] (c') -- (nM);
		\draw[->,dashed] (io) -- (o);
		\draw[->] (cN) -- (io);
		\draw[->,dashed] (cM) -- (o');
		\draw[->,dashed] (o') -- (io');
		\draw[->] (io') -- (nN);
		\draw[->] (cc') -- (r);
		\draw[->] (io') -- (r);
		\draw[->,dashed] (cM) -- (c);
		\draw[->] (o) -- (nM);
		\draw[->,dashed] (c) -- (cc);
		\draw[->,dashed] (i) -- (co);
		\draw[->,dotted] (fN) -- (co);
		\draw[] (io) edge[->,out=45, in=180, looseness=0.7] (i);
		\draw[] (cM) edge[->,out=45, in=270, looseness=0.7] (i);
		
	\end{tikzpicture}
	\caption{Relations between our cardinal characteristics}
	\label{fig:diagram}
\end{figure}

We note that because both $\fr s$ and $\cov(\mcal M)$ are lower bounds for $\ssn_{cc}$, that it is consistent that $\fr s<\ssn_{cc}$ or that $\cov(\mcal M)<\ssn_{cc}$. The first holds in the Cohen model, whereas the second holds in the Mathias model (see e.g.\ \cite[\S~11]{Blass} for definitions of these models). Several open questions with regard to consistency remain.

Firstly, many of our subseries numbers were equal to $\ssn_c$, specifically $\ssn_c=\ssn_f=\ssn_e^\epsilon=\ssn_l^{(0,1)}$. Simultaneously, many other subseries numbers were shown to be equal to $\fr d$, specifically $\fr d=\ssn_{ac}=\ssn_c^{ui}=\ssn_f^{ac}$.

Since $\cov(\mcal M)\leq\ssn_c\leq\fr d$ and consistently $\cov(\mcal M)<\fr d$ (e.g.\ in the Miller model), a natural question is:
\begin{qst}
	Which of $\ssn_c<\fr d$ and $\cov(\mcal M)<\ssn_c$ is consistent?
\end{qst}

Natural forcing notions to consider, are the Miller, Mathias and Laver forcing, all of which increase $\fr d$ without increasing $\cov(\mcal M)$. Another candidate is random forcing, which can be used to force $\non(\mcal N)<\cov(\mcal N)$ without affecting the cardinality of $\fr d$ from the ground model.

The situation for $\ssn_{cc}$ is similar. If $\ssn_{cc}<\fr d$ is consistent, then the Miller model would be an excellent candidate to check, since this model also has $\fr s=\aleph_1$. On the other hand, if $\ssn_{cc}=\fr d$ holds in the Miller model, this could indicate that $\ssn_{cc}=\fr d$ is provable in $\msf{ZFC}$.

\begin{qst}
	Is $\ssn_{cc}<\fr d$ consistent? Is $\max\ac{\fr s,\cov(\mcal M)}<\ssn_{cc}$ consistent?
\end{qst}

Naturally the questions above could be asked for the dual cardinal characteristics $\ssn_{cc}^\bot$, $\ssn_c^\bot$ and their relationship to $\non(\mcal M)$ and $\fr b$.


\subsection*{Acknowledgements}

A small part of this paper consists of the original work from my Master's thesis \cite{VlugtMThesis} (particularly the definition of $\ssn_c$ and $\ssn_e$). I would like to thank J\"{o}rg Brendle, who supervised that thesis, for valuable discussions, comments and suggestions, and Jakob Kellner and Martin Goldstern for providing support in Vienna. 

\printbibliography

\end{document}